\documentclass[12pt]{amsart}

\usepackage{amsmath}
\usepackage{amsfonts}
\usepackage{amssymb}
\usepackage{amsthm}
\usepackage{hyperref}
\usepackage{tikz}
\input xy 
\xyoption{all}
\usepackage{enumerate}
\usepackage{cleveref}
\usepackage{mathrsfs}
\usepackage[top=0.9in, bottom=1.15in, left=1.1in, right=1.1in]{geometry}
 \usepackage{hyperref}
\hypersetup{colorlinks=true,linkcolor=blue,citecolor=magenta}
\newtheorem{theorem}{Theorem}

\newtheorem{definition}[theorem]{Definition}
\newtheorem*{remark}{Remark}

\Crefname{conjecture}{Conjecture}{Conjectures}

\theoremstyle{plain}
\newtheorem{example}[theorem]{Example}
\theoremstyle{plain}

\numberwithin{equation}{section}

\author{Robert Schneider}

\address{Department of Mathematical Sciences\newline
Michigan Technological University\newline
Houghton, Michigan 49931, U.S.A.}
\email{robertsc@mtu.edu}

\title{Algebra of overpartitions}
 
\begin{document}
\begin{abstract}
In a 2022 paper, Dawsey, Just and the present author prove that the set of integer partitions, taken as a monoid under a partition multiplication operation I defined in my Ph.D. work, is isomorphic to the positive integers as a monoid under integer multiplication. In this note, I extend partition multiplication to the set of overpartitions, which are of much  interest in  partition theory. I prove  the overpartitions form an Abelian group under partition multiplication. Moreover, the overpartitions and the positive rational numbers are isomorphic as multiplicative  groups. I then prove further overpartition isomorphisms and discuss  approaches to a ring theory of overpartitions. 
\end{abstract}

\maketitle

\section{Introduction: concepts and notations}
\subsection{Partition multiplication}
In \cite{Robert_bracket}, the present author defines a multiplication operation on integer partitions, concatenation, which gives rise to a multiplicative theory of partitions analogous to multiplicative number theory in many respects; featuring partition-theoretic analogues of arithmetic functions, divisor sums, Dirichlet series, 
 and other topics studied in multiplicative number theory (see e.g.  \cite{Robert_zeta, Schneider_PhD}). In \cite{supernorm}, Dawsey, Just and I prove that  the  integer partitions  (under partition multiplication) and the positive integers (under multiplication in $\mathbb Z^+$) are isomorphic as monoids;   see  Figure \ref{figure1}.

Overpartitions, first described by MacMahon \cite{MacMahonI}, are  important objects of study in  partition theory \cite{over1, over3,  over2, Lovejoy}. Corteel and Lovejoy write, ``the theory of basic hypergeometric series contains a wealth of information about overpartitions, [and] many theorems and techniques for ordinary partitions have analogues for overpartitions.''\cite{Lovejoy} In this note, I prove   the set of {overpartitions} (under the same partition multiplication operation) and the {\it positive rational} numbers (under multiplication in $\mathbb Q^+$) are isomorphic as {\it groups}.

Let $\mathcal P$ denote the set of {\it integer partitions} (see \cite{Andrews}). 
We will denote a generic nonempty partition by $\lambda=(\lambda_1,\lambda_2,\dots,$ $\lambda_r),\ \lambda_i \in\mathbb Z^+,\   \lambda_1\geq\lambda_2\geq\dots\geq\lambda_r\geq 1$, with $\emptyset\in\mathcal P$  the {\it empty partition}. Let $\ell(\lambda):=r$ denote the partition {\it length} (number of parts). Let $m_i=m_i(\lambda)\geq 0$ denote the {\it multiplicity} (frequency) of $i$ as a part of the partition. Let $|\lambda|:=\lambda_1+\lambda_2+\dots+\lambda_r$ denote the partition {\it size} (sum of parts);   the {\it partition function} $p(n)$ counts the integer   partitions of size equal to $n$. 
We define $\ell(\emptyset)=m_i(\emptyset)=|\emptyset|=0.$ 

In this note, we will also make use of the {\it part-multiplicity} notation for partitions,  $$\lambda = \left<1^{m_1} 2^{m_2}3^{m_3}\cdots i^{m_i}\cdots\right>\in \mathcal P$$ where each part $i\in\mathbb Z^+$ has its multiplicity $m_i\geq 0$ as a superscript, allowing only finitely many nonzero multiplicities $m_i$ (and we omit parts that have multiplicity zero); with $\emptyset =\left< 1^0\  2^0\  3^0  \cdots i^0  \cdots \right>$. For instance, we will rewrite $(7, 5, 5, 2, 2, 2, 1)= \left<1^{1} \  2^{3}\  5^{2}\  7^{1}\right>$.

In \cite{Robert_bracket}, I define a {\it partition multiplication} operation equivalent to the following. For integer partitions $\lambda = \left<1^{m_1} 2^{m_2}3^{m_3}\cdots i^{m_i}\cdots\right>,\  \gamma = \left<1^{n_1} 2^{n_2}3^{n_3}\cdots i^{n_i}\cdots\right>,$ $m_i\geq 0, n_i \geq 0,$ 
define their product $\lambda \cdot \gamma\in \mathcal P$ by concatenation,  summing the corresponding multiplicities:
\begin{equation}\lambda\cdot  \gamma := \left<1^{m_1+n_1} 2^{m_2+n_2}3^{m_3+n_3}\cdots i^{m_i+n_i}\cdots\right>\in\mathcal P.\end{equation}

Thus the empty partition $\emptyset$ is the multiplicative identity, and $(\mathcal P, \  \cdot\  )$ is a monoid. Since  addition of   multiplicities is commutative, partition multiplication is commutative.  

\begin{figure}
\begin{center}
	    \begin{tikzpicture}
	        \begin{scope}[scale=1.1]
	            \node at (0,0){$\emptyset$};

	            \node at (-2,1){$\left<1^1\right>$};
	            \node at (0,1){$\left<2^1\right>$};
	            \node at (2,1){$\left<3^1\right>$};

	            \node at (-3.5,2){$\left<1^2\right>$};
	            \node at (-2,2){$\left<1^1  2^1\right>$};
	            \node at (-.5,2){$\left<2^2\right>$};
	            \node at (.5,2){$\left<1^1  3^1\right>$};
	            \node at (2,2){$\left<2^1  3^1\right>$};
	            \node at (3.5,2){$\left<3^2\right>$};

	            \node at (-4.5,3.75){$\left<1^3\right>$};
	            \node at (-3.5,4){$\left<1^2  2^1\right>$};
	            \node at (-2.5,3.75){$\left<1^2 3^1\right>$};
	            \node at (-1.5,4){$\left<1^1  2^2\right>$};
	            \node at (-.5,3.75){$\left<2^3\right>$};
	            \node at (0.5,4){$\left<2^2  3^1\right>\  $};
	            \node at (1.5,3.75){$\left<1^1  2^1  3^1\right>$};
	            \node at (2.5,4){$\  \left<1^1  3^2\right>$};
	            \node at (3.5,3.75){$\left<2^1  3^2\right>$};
	            \node at (4.5,4){$\left<3^3\right>$};

	            \draw[shorten <=.1in,shorten >=.1in](0,0)--(0,1);
	            \draw[shorten <=.1in,shorten >=.1in](0,0)--(-2,1);
	            \draw[shorten <=.1in,shorten >=.1in](0,0)--(2,1);
	            \draw[shorten <=.1in,shorten >=.1in](0,0)--(4,1);
	            \draw[shorten <=.1in,shorten >=.1in](-2,1)--(-2,2);
	            \draw[shorten <=.1in,shorten >=.1in](0,1)--(-.5,2);
	            \draw[shorten <=.2in,shorten >=.2in](0,1)--(-2,2);
	            \draw[shorten <=.1in,shorten >=.1in](2,1)--(2,2);
	            \draw[shorten <=.2in,shorten >=.2in](-2,1)--(-3.5,2);
	            \draw[shorten <=.2in,shorten >=.2in](-2,1)--(.5,2);
	            \draw[shorten <=.2in,shorten >=.2in](0,1)--(2,2);
	            \draw[shorten <=.2in,shorten >=.2in](2,1)--(.5,2);
	            \draw[shorten <=.2in,shorten >=.2in](2,1)--(3.5,2);

	            \draw[shorten <=.2in,shorten >=.2in](-3.5,2)--(-4.5,3.75);
	            \draw[shorten <=.2in,shorten >=.2in](-3.5,2)--(-3.5,4);
	            \draw[shorten <=.2in,shorten >=.2in](-3.5,2)--(-2.5,3.75);

	            \draw[shorten <=.2in,shorten >=.2in](-2,2)--(-3.5,4);
	            \draw[shorten <=.2in,shorten >=.2in](-2,2)--(-1.5,4);
	            \draw[shorten <=.2in,shorten >=.2in](-2,2)--(1.5,3.75);

	            \draw[shorten <=.2in,shorten >=.2in](-.5,2)--(-1.5,4);
	            \draw[shorten <=.2in,shorten >=.2in](-.5,2)--(-.5,3.75);
	            \draw[shorten <=.2in,shorten >=.2in](-.5,2)--(.5,4);

	            \draw[shorten <=.2in,shorten >=.2in](.5,2)--(-2.5,3.75);
	            \draw[shorten <=.2in,shorten >=.2in](.5,2)--(1.5,3.75);
	            \draw[shorten <=.2in,shorten >=.2in](.5,2)--(2.5,4);

	            \draw[shorten <=.2in,shorten >=.2in](2,2)--(.5,4);
	            \draw[shorten <=.2in,shorten >=.2in](2,2)--(1.5,3.75);
	            \draw[shorten <=.2in,shorten >=.2in](2,2)--(3.5,3.75);

	            \draw[shorten <=.2in,shorten >=.2in](3.5,2)--(2.5,4);
	            \draw[shorten <=.2in,shorten >=.2in](3.5,2)--(3.5,3.75);
	            \draw[shorten <=.2in,shorten >=.2in](3.5,2)--(4.5,4);

	            \node at (4,1){$\ldots$};

	            \node at (-4.5,4.5){$\vdots$};
	            \node at (-3.5,4.5){$\vdots$};
	            \node at (-2.5,4.5){$\vdots$};
	            \node at (-1.5,4.5){$\vdots$};
	            \node at (-.5,4.5){$\vdots$};
	            \node at (4.5,4.5){$\vdots$};
	            \node at (3.5,4.5){$\vdots$};
	            \node at (2.5,4.5){$\vdots$};
	            \node at (1.5,4.5){$\vdots$};
	            \node at (.5,4.5){$\vdots$};

	        \end{scope}
	    \end{tikzpicture}
	\end{center}

		\begin{center}
	    \begin{tikzpicture}
	        \begin{scope}[scale=1.1]
	            \node at (0,0){1};

	            \node at (-2,1){2};
	            \node at (0,1){3};
	            \node at (2,1){5};

	            \node at (-3.5,2){4};
	            \node at (-2,2){6};
	            \node at (-.5,2){9};
	            \node at (.5,2){10};
	            \node at (2,2){15};
	            \node at (3.5,2){25};

	            \node at (-4.5,3.75){8};
	            \node at (-3.5,4){12};
	            \node at (-2.5,3.75){20};
	            \node at (-1.5,4){18};
	            \node at (-.5,3.75){27};
	            \node at (0.5,4){45};
	            \node at (1.5,3.75){30};
	            \node at (2.5,4){50};
	            \node at (3.5,3.75){75};
	            \node at (4.5,4){125};

	            \draw[shorten <=.1in,shorten >=.1in](0,0)--(0,1);
	            \draw[shorten <=.1in,shorten >=.1in](0,0)--(-2,1);
	            \draw[shorten <=.1in,shorten >=.1in](0,0)--(2,1);
	            \draw[shorten <=.1in,shorten >=.1in](0,0)--(4,1);
	            \draw[shorten <=.1in,shorten >=.1in](-2,1)--(-2,2);
	            \draw[shorten <=.1in,shorten >=.1in](0,1)--(-.5,2);
	            \draw[shorten <=.2in,shorten >=.2in](0,1)--(-2,2);
	            \draw[shorten <=.1in,shorten >=.1in](2,1)--(2,2);
	            \draw[shorten <=.2in,shorten >=.2in](-2,1)--(-3.5,2);
	            \draw[shorten <=.2in,shorten >=.2in](-2,1)--(.5,2);
	            \draw[shorten <=.2in,shorten >=.2in](0,1)--(2,2);
	            \draw[shorten <=.2in,shorten >=.2in](2,1)--(.5,2);
	            \draw[shorten <=.2in,shorten >=.2in](2,1)--(3.5,2);

	            \draw[shorten <=.2in,shorten >=.2in](-3.5,2)--(-4.5,3.75);
	            \draw[shorten <=.2in,shorten >=.2in](-3.5,2)--(-3.5,4);
	            \draw[shorten <=.2in,shorten >=.2in](-3.5,2)--(-2.5,3.75);

	            \draw[shorten <=.2in,shorten >=.2in](-2,2)--(-3.5,4);
	            \draw[shorten <=.2in,shorten >=.2in](-2,2)--(-1.5,4);
	            \draw[shorten <=.2in,shorten >=.2in](-2,2)--(1.5,3.75);

	            \draw[shorten <=.2in,shorten >=.2in](-.5,2)--(-1.5,4);
	            \draw[shorten <=.2in,shorten >=.2in](-.5,2)--(-.5,3.75);
	            \draw[shorten <=.2in,shorten >=.2in](-.5,2)--(.5,4);

	            \draw[shorten <=.2in,shorten >=.2in](.5,2)--(-2.5,3.75);
	            \draw[shorten <=.2in,shorten >=.2in](.5,2)--(1.5,3.75);
	            \draw[shorten <=.2in,shorten >=.2in](.5,2)--(2.5,4);

	            \draw[shorten <=.2in,shorten >=.2in](2,2)--(.5,4);
	            \draw[shorten <=.2in,shorten >=.2in](2,2)--(1.5,3.75);
	            \draw[shorten <=.2in,shorten >=.2in](2,2)--(3.5,3.75);

	            \draw[shorten <=.2in,shorten >=.2in](3.5,2)--(2.5,4);
	            \draw[shorten <=.2in,shorten >=.2in](3.5,2)--(3.5,3.75);
	            \draw[shorten <=.2in,shorten >=.2in](3.5,2)--(4.5,4);

	            \node at (4,1){$\ldots$};

	            \node at (-4.5,4.5){$\vdots$};
	            \node at (-3.5,4.5){$\vdots$};
	            \node at (-2.5,4.5){$\vdots$};
	            \node at (-1.5,4.5){$\vdots$};
	            \node at (-.5,4.5){$\vdots$};
	            \node at (4.5,4.5){$\vdots$};
	            \node at (3.5,4.5){$\vdots$};
	            \node at (2.5,4.5){$\vdots$};
	            \node at (1.5,4.5){$\vdots$};
	            \node at (.5,4.5){$\vdots$};

	        \end{scope}
	    \end{tikzpicture}
	\end{center}
	\caption{Illustration of monoid isomorphism reproduced from \cite{supernorm}: partitions ordered by multiset inclusion, positive integers ordered by divisibility}
	\label{figure1}
\end{figure}

\subsection{Multiplication of overpartitions}
Let $\mathcal O$ denote the set of {\it overpartitions} (see  \cite{Lovejoy}), partitions where the first instance of a given part size to occur, can either be marked with an overline (bar) or left unmarked; with $\emptyset$ the empty overpartition. E.g., partition $(3, 2, 2, 2, 1, 1)$ gives rise to the overpartitions $(\overline{3}, 2, 2, 2, 1, 1),$ $(3, \overline{2}, 2, 2, 1, 1),$ $(3, 2, 2, 2, \overline{1}, 1),$ $(\overline{3}, \overline{2}, 2, 2, 1, 1), (\overline{3}, 2, 2, 2, \overline{1}, 1), (3, \overline{2}, 2, 2, \overline{1}, 1),$ $(\overline{3}, \overline{2}, 2, 2, \overline{1}, 1)$, and $(3, 2, 2, 2, 1, 1)$ itself. We extend  the partition statistics  above to overpartitions, with identical meanings; i.e., for $\alpha\in\mathcal O$,  let $\ell(\alpha)$ denote the number of parts, $|\alpha|$ denote the sum of parts, etc. 

We  extend part-multiplicity notation to $\mathcal O$, as follows.

\begin{definition}\label{def1}
For overpartition $\alpha \in \mathcal O$, we write
$$\alpha= \left<1^{\mu_1} 2^{\mu_2}3^{\mu_3}\cdots i^{\mu_i}\cdots\right>, \  \  \  \mu_i\in \mathbb Z,$$ 
with multiplicity  $\mu_i$ of part $i\in\mathbb Z^+$
 labeled as   negative  if and only if part $i$ has an overline.
\footnote{We use Greek letters for overpartition multiplicities, to help distinguish overpartitions from partitions.} 
\end{definition}

\begin{remark}
 Parts-multiplicity notation for overpartitions originated with Michigan Technological University Ph.D. students Philip Cuthbertson and Hunter Waldron, during the course of my graduate group theory class (Fall 2023). Cuthbertson and Waldron both independently showed me that negative multiplicities  can be interpreted as overlines.\footnote{This refines the concepts of antipartitions and rational partitions   I introduced in \cite{Schneider_PhD}, Appendix B.3.1.} 
\end{remark}

One can directly  verify the bijection between overlined parts and negative multiplicities. If $\lambda\in \mathcal P$ is the partition one  forms from the parts of $\alpha\in \mathcal O$ whose multiplicities are   positive, and $\gamma\in \mathcal P$ the partition one  forms from the parts with negative multiplicities, we also define a {\it rational partition} form for the overpartition  by writing 
$\alpha = {\lambda}/{\gamma}\in \mathcal O.$ 

We  now extend the partition multiplication operation to overpartitions.

\begin{definition}\label{def2}
For overpartitions $\alpha= \left<1^{\mu_1} 2^{\mu_2}3^{\mu_3}\cdots i^{\mu_i}\cdots\right>, \beta= \left<1^{\nu_1} 2^{\nu_2}3^{\nu_3}\cdots i^{\nu_i}\cdots\right>$, with $\mu_i, \nu_i \in \mathbb Z$, we define their product by the overpartition
$$\alpha \cdot \beta = \left<1^{\mu_1+\nu_1} 2^{\mu_2+\nu_2}3^{\mu_3+\nu_3}\cdots i^{\mu_i+\nu_i}\cdots\right>\in \mathcal O,$$ 
with $\mu_i+\nu_i<0$   if and only if part $i\in \mathbb Z^+$ has an overline  in  $\alpha\cdot \beta$.
\end{definition}
Just as in the partition-theoretic case, the empty overpartition $\emptyset$ is the multiplicative identity, $(\mathcal O,\  \cdot\  )$ is also a monoid, and partition multiplication in $\mathcal O$ is commutative.

\section{Group theory of overpartitions}

\subsection{Overpartitions group structure}
We have  developed our conceptual background sufficiently to state the central proposition of this note.\footnote{Theorem \ref{thm} is an equivalent statement to Theorem B.3.2 in \cite{Schneider_PhD}, which concerns ``rational partitions''.}

\begin{theorem}\label{thm}
The overpartitions $(\mathcal O, \ \cdot\  )$ with partition multiplication 
form an Abelian group, with identity element $\emptyset$.
\end{theorem}

\begin{proof}
The theorem follows  from Definitions \ref{def1}
 and \ref{def2}, since the product of overpartitions depends on the addition of  multiplicities $\mu_i, \nu_i\in\mathbb Z$, and $(\mathbb Z, +  )$ is an Abelian group. 
\end{proof}

Knowing that $(\mathcal O,\  \cdot\  )$ is an Abelian group allows us to identify algebraic structures connected to the set of overpartitions. It follows immediately from Definition \ref{def1} that 
\begin{equation}\mathcal O \  \cong \  \mathbb Z \oplus \mathbb Z\oplus \mathbb Z \oplus \dots,\end{equation}
with partition multiplication on the left side of the map, and integer addition on the right. We can use methods from algebra to prove further isomorphisms.

\begin{theorem}\label{thm2}
The group of overpartitions $(\mathcal O,\  \cdot\  )$ with partition multiplication 
is isomorphic   to the positive rationals $(\mathbb Q^{+},\ \cdot\  )$ with rational number multiplication.
\end{theorem}

\begin{proof}
Recall the {\it supernorm} statistic $\widehat{N}(\lambda)$ 
defined on partition $\lambda = \left<1^{m_1} 2^{m_2}3^{m_3}\cdots i^{m_i}\cdots\right>$  in \cite{supernorm}\footnote{See \cite{Lagarias} for a deep subsequent study of the supernorm map by J. Lagarias.}, by 
$$\widehat{N}(\lambda) := p_1^{m_1} p_2^{m_2} p_3^{m_3}\cdots p_i^{m_i}\cdots,$$
where $p_i \in \mathbb P$ is the $i$th prime number, viz. $p_1=2, p_2=3, p_3=5,$ etc. In \cite{supernorm}, the map $\widehat{N}$ induces an isomorphism of monoids between $\mathcal P$ and $\mathbb Z^+$. We will extend the supernorm to a map on overpartitions. For $\alpha= \left<1^{\mu_1} 2^{\mu_2}3^{\mu_3}\cdots i^{\mu_i}\cdots\right>\in\mathcal O$, define the {\it overpartition supernorm} map  $\widehat{N}_{\mathcal O}\colon \mathcal O\to\mathbb Q^+$ by
\begin{equation}
\widehat{N}_{\mathcal O}(\alpha) := p_1^{\mu_1} p_2^{\mu_2} p_3^{\mu_3}\cdots p_i^{\mu_i}\cdots=\frac{\widehat{N}(\lambda)}{\widehat{N}(\gamma)},
\end{equation}  
with $\alpha=\lambda/\gamma,\  \lambda, \gamma \in \mathcal P,$ being the rational form of the overpartition. 

For $\alpha, \beta\in\mathcal O$, we have  $\widehat{N}_{\mathcal O}(\alpha\cdot\beta)=\widehat{N}_{\mathcal O}(\alpha)\widehat{N}_{\mathcal O}(\beta)$ with $\widehat{N}_{\mathcal O}(\emptyset)=1$, the  identity in $(\mathbb Q^+,\  \cdot\  )$, thus the overpartition supernorm is a homomorphism of groups. The image of $\widehat{N}_{\mathcal O}$ is $\mathbb Q^+$, and the kernel is $\{\emptyset\}\subset \mathcal O$. The kernel being the identity  implies the map is injective, and we showed $\widehat{N}_{\mathcal O}$ is surjective, thus the map is an isomorphism. 
\end{proof}

Recalling    $\mathcal P$ is isomorphic to $\mathbb Z^+$ as multiplicative {monoids}  \cite{supernorm},    
Theorem \ref{thm2} tells us that under their respective multiplication operations,   {\it the overpartitions $\mathcal O$ extend the  monoid $\mathcal P$ to a group, just as  the positive rationals $\mathbb Q^+$ extend the  monoid $\mathbb Z^+$}.\footnote{Thank you to H. Waldron for a useful conversation on this subject.} (See Figure \ref{figure2})

\begin{center}
\begin{figure}
	$$    \xymatrix{\ar @{} [dr] |{ }
 {\mathcal P}\ar[d]_{\text{ }}  \ar@{<->}[r]^{\widehat{N}}        & {\mathbb Z^+} \ar[d]^{\text{ }}   \\ {\mathcal O} \ar@{<->}[r]^{\widehat{N}_{\mathcal O}} & {\mathbb Q^+}
     }
$$		
	\caption{$\mathcal O$ extends $\mathcal P$ just as $\mathbb Q^+$ extends $\mathbb Z^+$,  by adjoining multiplicative inverses, with  $\widehat{N}$ and $  \widehat{N}_{\mathcal O}$ inducing isomorphisms} 
	\label{figure2}
\end{figure}
\end{center}

\subsection{Further isomorphisms}
In the proof of Theorem \ref{thm2} in the preceding section, we used that $\widehat{N}_{\mathcal O}$ is a group homomorphism. Note also that the ``multiplicity of $k$'' map $\mu_k\colon (\mathcal O,\  \cdot\  ) \to (\mathbb Z,\  +  )$  is a group homomorphism, since for $\alpha, \beta\in\mathcal O$, we have $\mu_k(\alpha\cdot \beta)=\mu_k(\alpha)+\mu_k(\beta),\  \mu_k(\emptyset)=0$. Partition statistics  extend to overpartition homomorphisms. 

Since all subgroups of an Abelian group are normal, the group isomorphism theorems can readily be applied to overpartitions. We can identify other  overpartition isomorphisms if we identify further statistics on $\mathcal O$.\footnote{Thank you to P. Cuthbertson for a useful discussion of overpartition statistics.} For example, for $\alpha= \left<1^{\mu_1} 2^{\mu_2}3^{\mu_3}\cdots i^{\mu_i}\cdots\right>\in \mathcal O, \  \mu_i\in \mathbb Z,$ let ${|\alpha|}_{\mathcal O}$ be called the {\it oversize} of the overpartition, 
and let ${\ell}_{\mathcal O}(\alpha)$ be    the {\it overlength}, 
with the following definitions:\footnote{One might  also define an {\it overnorm} $ {{N}}_{\mathcal O}(\alpha) := \prod_{i\geq 1}  i^{\mu_i}\  \in\  \mathbb Q^+$ to generalize the partition norm \cite{SS_norm}.} 
\begin{equation}{|\alpha|_{\mathcal O}} := \sum_{i\geq 1} i\cdot \mu_i\  \in\  \mathbb Z,  \  \  \  \  \  \  \  \  \  {\ell}_{\mathcal O}(\alpha)  := \sum_{i\geq 1}  \mu_i\  \in\  \mathbb Z. 
\end{equation}
Since for $\alpha, \beta \in \mathcal O$, we have $|\alpha\cdot \beta|_{\mathcal O}=|\alpha|_{\mathcal O}+|\beta|_{\mathcal O}$ and   ${\ell}_{\mathcal O}(\alpha\cdot \beta)={\ell}_{\mathcal O}(\alpha)+{\ell}_{\mathcal O}(\beta)$, 
with each map taking the multiplicative identity  $\emptyset\in\mathcal O$ to the additive identity   $0 \in \mathbb Z^+$, then in the context of overpartition group theory these statistics represent homomorphisms: 
$${|*  |_{\mathcal O}}\colon  (\mathcal O,\  \cdot\  )  \to   (\mathbb Z,\  +  ),  \   \  \  \  \   \  \  \  {\ell}_{\mathcal O}\colon (\mathcal O,\  \cdot\  )  \to   (\mathbb Z,\  +  ).$$ 

Natural homomorphisms such as these induce subsets of $\mathcal O$ that can prove useful. Recall the rational partition notation $\alpha=\lambda/\gamma,\  \lambda, \gamma\in \mathcal P,  $ for  overpartition $\alpha$ defined above. Let
\begin{flalign}
\mathcal O_{|*|}&:=\  \{\lambda/\gamma\in\mathcal O  :\  \lambda, \gamma \in \mathcal P, \  |\lambda|=|\gamma|\},\\
\mathcal O_{\ell}\  &:=\ \{\lambda/\gamma\in\mathcal O  :\  \lambda, \gamma \in \mathcal P, \ \ell(\lambda)=\ell(\gamma)\}.
\end{flalign}
Noting that both subsets $\mathcal O_{|*|}$ and $\mathcal O_{\ell}$  are closed under partition multiplication, and for each overpartition $\alpha=\lambda/\gamma$ the  multiplicative inverse $\alpha^{-1}=\gamma/\lambda$ is also in the subset, then  $\mathcal O_{|*|}$ and $\mathcal O_{\ell}$ are  (normal) subgroups of $\mathcal O$.

\begin{example}\label{thmx}
The quotient group $(\mathcal O/\mathcal O_{|*|}\  ,\  \cdot\  )$ with partition multiplication 
is isomorphic   to the  integers $(\mathbb Z, +  )$ with  addition.
\end{example}

\begin{proof}
Overnorm $|*|_{\mathcal O}$ is a group homomorphism  with kernel $\mathcal O_{|*|}\triangleleft\mathcal O$ and image $(\mathbb Z, +  )$. Then the result follows from the First Isomorphism Theorem.  
\end{proof}

\begin{remark}
    The set $\mathcal O_{|*|}$ makes connections in the theory of prefabs.\footnote{A {\it prefab} is a   combinatorial structure introduced by Bender and Goldman \cite{prefab} that generalizes broad classes of combinatorial objects including integer partitions, plane partitions and rooted unlabeled forests.} For $n\geq 1$, define $\mathcal O_{|*|}(n):= \{\alpha= \lambda / \gamma \in \mathcal O_{|*|}\  :\   |\lambda| = |\gamma| =n\}$; then $\mathcal O_{|*|}(n)$ is  equivalent to the prefab consisting of ordered pairs 
    of partitions having equal size $n\geq 1$, which is  studied in \cite{Corteel}. Many facts about $\mathcal O_{|*|}(n)$ are proved in \cite{Corteel};  e.g.   it is  equivalent   to Proposition 1 of \cite{Corteel} that 
    \begin{equation*}
    \#\mathcal O_{|*|}(n)\  =\  p(n)^2-p(n-1)^2-p(n-2)^2+p(n-5)^2+p(n-7)^2-\dots + (-1)^{\left \lceil j / 2 \right \rceil} p(n-\pi_j)^2 + \dots,
    \end{equation*}
where  $\pi_j$ denotes the $j$th pentagonal number and $\lceil x \rceil,\  x \in \mathbb R,$ is the usual ceiling function.\footnote{Thank you to B. Hopkins for pointing out the reference \cite{Corteel} to the author.} 
\end{remark}


\begin{example}
The quotient group $(\mathcal O/\mathcal O_{\ell}\  ,\  \cdot\  )$ with partition multiplication 
is isomorphic   to the  integers $(\mathbb Z, +  )$ with  addition.
\end{example}

\begin{proof} Overlength $\ell_{\mathcal O}$ is  a group homomorphism,  with kernel $\mathcal O_{\ell}\  \triangleleft \  \mathcal O$ and image $(\mathbb Z, +  )$. The result follows from the First Isomorphism Theorem, as with the previous proof. 
\end{proof}

Let $S\subseteq \mathbb Z^+$. To state another structural result, define the following subsets of $\mathcal O$:
\begin{flalign}
\mathcal O_{S}\  &:=\  \{\alpha=\left<1^{\mu_1} 2^{\mu_2}\cdots i^{\mu_i}\cdots\right>  :\  \mu_i\neq 0\  \text{only if}\  i\in S\},\\
\mathcal O_{\mathbb Z^+\backslash S}&:=\  \{\alpha=\left<1^{\mu_1} 2^{\mu_2}\cdots i^{\mu_i}\cdots\right>  :\  \mu_i= 0\  \text{if}\  i\in S\}.
\end{flalign}
I.e., $\mathcal O_{S}$ denotes overpartitions with every part in $S$, and $\mathcal O_{\mathbb Z^+\backslash S}$ those   with no part in $S$. Noting the subsets $\mathcal O_{S}, \mathcal O_{\mathbb Z^+\backslash S}$  are closed under partition multiplication, and for every overpartition $\alpha=\left<1^{\mu_1} 2^{\mu_2}\cdots i^{\mu_i}\cdots\right>$ in each subset, the  multiplicative inverse $\alpha^{-1}=\left<1^{-\mu_1} 2^{-\mu_2}\cdots i^{-\mu_i}\cdots\right>$ is   in the subset, then  $\mathcal O_{S}$ and $\mathcal O_{\mathbb Z^+\backslash S}$ are  (normal) subgroups of $\mathcal O$.


\begin{example} 
The quotient group $(\mathcal O/\mathcal O_{S}\  ,\  \cdot\  )$ with partition multiplication 
is isomorphic   to $(\mathcal O_{\mathbb Z^+\backslash S}\  ,\ \cdot\  )$ 
with partition multiplication.
\end{example}

\begin{proof}
Let us define a map $\phi_S\colon \mathcal O\to\mathcal O_{\mathbb Z^+\backslash S}$ such that $\phi_S(\alpha)=\alpha$ if $\alpha \in \mathcal O_{\mathbb Z^+\backslash S}$, and otherwise, $\phi_S(\alpha)\in \mathcal O_{\mathbb Z^+\backslash S}$ deletes all parts from  $\alpha\in \mathcal O$ that are elements of $S\subseteq \mathbb Z^+$. We have for $\alpha, \beta \in \mathcal O$ that $\phi_S(\alpha \cdot \beta)=\phi_S(\alpha)\cdot\phi_S(\beta)\in \mathcal O_{\mathbb Z^+\backslash S}$ with  partition multiplication   on both sides of the map, a homomorphism of groups. The kernel of $\phi_S$ is $\mathcal O_S$ and the image is $(\mathcal O_{\mathbb Z^+\backslash S}, \  \cdot\  )$. Then the result follows from the First Isomorphism Theorem.
\end{proof}

Here is one more application of the First Isomorphism Theorem. For $m\in\mathbb Z$, let $\mathbb Z/m\mathbb Z$ denote the integers modulo $m$, and let 
\begin{equation}
\mathcal O_{m}\  :=\  \{\alpha\in\mathcal O :\  \ell_{\mathcal O}(\alpha)\in m\mathbb Z\}.
\end{equation}
The set $\mathcal O_m$ is closed under partition multiplication, and for each $\alpha\in \mathcal O_m$, note that $\ell_{\mathcal O}(\alpha^{-1})=-\ell_{\mathcal O}(\alpha)\in m\mathbb Z$; thus $\alpha^{-1}\in\mathcal O_m$.  Then $\mathcal O_m$ is a (normal) subgroup of $\mathcal O$.

\begin{example}
The quotient group $(\mathcal O/\mathcal O_{m}\  ,\  \cdot\  )$ with partition multiplication 
is isomorphic   to $(\mathbb Z/m\mathbb Z\  , +\  )$ 
with addition modulo $m$.
\end{example}

\begin{proof}
Let us define a map $\overline{\ell}_{\mathcal O}^m\colon \mathcal O\to \mathbb Z/m\mathbb Z$ such that $\overline{\ell}_{\mathcal O}^m(\alpha)$ is the congruence class of $\ell_{\mathcal O}(\alpha)\in \mathbb Z$, modulo $m$. Since $\overline{\ell}_{\mathcal O}^m(\alpha\cdot \beta)
=\overline{\ell}_{\mathcal O}^m(\alpha)+\overline{\ell}_{\mathcal O}^m(\beta)$ under addition in $(\mathbb Z/m\mathbb Z,+\  )$, $\overline{\ell}_{\mathcal O}^m(\emptyset)=0$, the map is a homomorphism of groups. The kernel of $\overline{\ell}_{\mathcal O}^m$ is $\mathcal O_m \triangleleft \mathcal O$, and the image is $(\mathbb Z/m\mathbb Z,+\  )$. Then the result follows from the First Isomorphism Theorem.
\end{proof}

\section{Toward a ring theory of overpartitions}

Looking to the future,  the ring theory of overpartitions is of   interest to the author. With my collaborator Ian Wagner, I   worked out  two different versions of such a theory that were noted in \cite{Schneider_PhD}, Appendix B, one based on matrix algebra and another using group rings; these included a cursory study of irreducibles that I would like to extend. I also shared preliminary notes in a different direction to the audience of Michigan Technological University's  online Partition and $q$-Series Seminar \cite{multipartitions}. It is outside the scope of this note to detail these ideas further. In  the aforementioned cases, I introduced {\it ad hoc} partition extensions such as ``antipartitions'' and ``rational partitions'' as the group elements, prior to having the equivalent interpretation of overpartitions from Cuthbertson and Waldron. 

I closed Appendix B of  my Ph.D. dissertation \cite{Schneider_PhD} with the following words: ``It is our goal to use this    ring structure to seek alternative proofs of partition bijections, Ramanujan-like congruences and other classical partition theorems, as well as to seek applications in Andrews’s theory of partition ideals \cite{Andrews}.'' I hope that new theorems and structures can be discovered in the sets $\mathcal P$ and $\mathcal O$ by applying the   tools of abstract algebra.



%
%
%
%
%
\section*{Acknowledgements}
I am   very  grateful to the  students in my Fall 2023 graduate algebra course at Michigan Technological University for useful discussions on these combinatorial-algebraic ideas: Aidan Botkin, Philip Cuthbertson,  Caleb Hiltunen, Cody McCarthy, Mason Moran,  Faria Tasnim and Hunter Waldron. I am   thankful to David Hemmer and Hunter Waldron for    corrections to an earlier draft of this paper; and  to Krishnaswami Alladi, George Andrews,  Madeline L. Dawsey, Philip Engel, Brian Hopkins, Matthew R. Just, Jeffrey Lagarias, Ken Ono, Andrew V. Sills, and Ian Wagner for     conversations that informed my work.  Furthermore, I am   thankful to the anonymous referee of this paper.


\begin{thebibliography}{BrStr}

\bibitem{Andrews} G. E. Andrews, {\it The Theory of Partitions}, Cambridge University Press, 1988.

\bibitem{over1} G. E. Andrews,  Singular overpartitions, {\it Int. J. Number Theory} {\bf 11.5} (2015): 1523-1533.


\bibitem{prefab} E. A. Bender and J. R. Goldman, Enumerative uses of generating functions, {\it Indiana Univ. Math. J.} {\bf 20} (1971), 753-765.

\bibitem{over3}  K. Bringmann,    and J. Lovejoy, Dyson's rank, overpartitions, and weak Maass forms, {\it International Mathematics Research Notices 2007} (2007): rnm063.



\bibitem{Lovejoy} S. Corteel,  and J. Lovejoy, Overpartitions, {\it Transactions of the American Mathematical Society} {\bf 356.4} (2004): 1623-1635.
 

\bibitem{Corteel} S. Corteel, C. D. Savage, H. S. Wilf,  and  D. Zeilberger, A Pentagonal Number Sieve, {\it Journal of Combinatorial Theory, Series A}, {\bf 82(2)} (1998): 186-192.



\bibitem{over2} M. D. Hirschhorn, and  J. A. Sellers, Arithmetic relations for overpartitions, {\it J. Combin. Math. Combin. Comput.} {\bf 53.65-73} (2005): 1.

\bibitem{Lagarias} J. C. Lagarias,  Characterizing the supernorm partition statistic, {\it The Ramanujan Journal} {\bf 63.1} (2024): 195-207.


\bibitem{supernorm} M. L. Dawsey, M. Just, and R. Schneider, A ``supernormal'' partition statistic, {\it Journal of Number Theory} {\bf 241} (2022): 120-141.

\bibitem{MacMahonI} P. A. MacMahon, {\it Combinatory Analysis, Volumes I and II}, Vol. 137, American Mathematical Society, 2001.


 \bibitem{Robert_zeta} R. Schneider, Partition zeta functions, {\it Research in Number Theory} {\bf 2(9)} (2016).


 \bibitem{Robert_bracket} R. Schneider, Arithmetic of partitions and the $q$-bracket operator, {\it Proceedings of the American Mathematical Society} {\bf 145(5)} (2017): 1953-1968.

\bibitem{Schneider_PhD} R. Schneider, {\it Eulerian Series, 
Zeta Functions and the Arithmetic of Partitions}, Ph.D. dissertation, Emory University, 2018.

\bibitem{multipartitions} 
R. Schneider, {Notes toward an algebra of multipartitions}, \url{https://pages.mtu.edu/~wjkeith/PartitionsSpecialtySeminar/Algebra_of_multipartitions.4.pdf} (27 Sept., 2024), Web.

\bibitem{SS_norm} R. Schneider, and A. V. Sills, The product of parts or “norm” of a partition, {\it Integers} {\bf 20A} (2020).



\end{thebibliography}
\end{document}